\documentclass[10pt,english,reqno]{amsart}
\usepackage[T1]{fontenc}
\usepackage[latin9]{inputenc}
\usepackage{textcomp}
\usepackage{amsthm}
\usepackage{amstext}
\usepackage{amssymb}

\makeatletter

\newcommand{\lyxmathsym}[1]{\ifmmode\begingroup\def\b@ld{bold}
  \text{\ifx\math@version\b@ld\bfseries\fi#1}\endgroup\else#1\fi}

\providecommand{\tabularnewline}{\\}

\numberwithin{equation}{section}
\numberwithin{figure}{section}
\theoremstyle{plain}
\newtheorem{thm}{\protect\theoremname}
  \theoremstyle{plain}
  \newtheorem{prop}[thm]{\protect\propositionname}
  \theoremstyle{plain}
  \newtheorem{cor}[thm]{\protect\corollaryname}
  \theoremstyle{remark}
  \newtheorem{rem}[thm]{\protect\remarkname}

\usepackage[matrix,arrow,tips,curve,ps]{xy}\SelectTips{cm}{10}
\usepackage{enumerate}%\input xypic
\usepackage{comment} 

\theoremstyle{definition}

\numberwithin{Theorem}{section} \numberwithin{equation}{section}
\def\CC{{\mathbb C}}
\def\DD{{\mathbb D}}

\def\FF{{\mathbb F}}

\def\HH{{\mathbb H}}

\def\PP{{\mathbb P}}
\def\QQ{{\mathbb Q}}
\def\RR{{\mathbb R}}

\def\ZZ{{\mathbb Z}}

\newcommand{\NS}{\operatorname{NS}}

\makeatother

\usepackage{babel}
  \providecommand{\corollaryname}{Corollary}
  \providecommand{\propositionname}{Proposition}
  \providecommand{\remarkname}{Remark}
\providecommand{\theoremname}{Theorem}

\begin{document}

\addtolength{\textwidth}{0mm}
\addtolength{\hoffset}{-0mm} 
\addtolength{\textheight}{0mm}
\addtolength{\voffset}{-0mm} 

\global\long\global\long\def\Alb{{\rm Alb}}
 \global\long\global\long\def\Aut{{\rm Aut}}
 \global\long\global\long\def\Jac{{\rm Jac}}
 \global\long\global\long\def\Hom{{\rm Hom}}
 \global\long\global\long\def\End{{\rm End}}
 \global\long\global\long\def\aut{{\rm Aut}}
 \global\long\global\long\def\NS{{\rm NS}}
 \global\long\global\long\def\SSm{{\rm S}}
 \global\long\global\long\def\psl{{\rm PSL}}
 \global\long\global\long\def\CC{\mathbb{C}}
 \global\long\global\long\def\BB{\mathbb{B}}
 \global\long\global\long\def\PP{\mathbb{P}}
 \global\long\global\long\def\QQ{\mathbb{Q}}
 \global\long\global\long\def\RR{\mathbb{R}}
 \global\long\global\long\def\FF{\mathbb{F}}
 \global\long\global\long\def\DD{\mathbb{D}}
 \global\long\global\long\def\NN{\mathbb{N}}
 \global\long\global\long\def\ZZ{\mathbb{Z}}
 \global\long\global\long\def\HH{\mathbb{H}}
 \global\long\global\long\def\gal{{\rm Gal}}
 \global\long\global\long\def\OO{\mathcal{O}}
 \global\long\global\long\def\pP{\mathfrak{p}}
 \global\long\global\long\def\pPP{\mathfrak{P}}
 \global\long\global\long\def\qQ{\mathfrak{q}}
 \global\long\global\long\def\mm{\mathcal{M}}
 \global\long\global\long\def\aaa{\mathfrak{a}}

\title{On the cohomology of Stover Surface}

\author{Amir D\v{z}ambi\'c, Xavier Roulleau}

\maketitle

\begin{abstract}
We study a surface discovered by Stover which is the surface with
minimal Euler number and maximal automorphism group among smooth arithmetic
ball quotient surfaces. We study the natural map $\wedge^{2}H^{1}(S,\CC)\to H^{2}(S,\CC)$
and we discuss the problem related to the so-called Lagrangian surfaces.
We obtain that this surface $S$ has maximal Picard number and has
no higher genus fibrations. We compute that its Albanese variety $A$
is isomorphic to $(\CC/\ZZ[\alpha])^{7}$, for $\alpha=e^{2i\pi/3}$.
\end{abstract}

\section{Introduction}

By the recent work of M. Stover \cite{Stover}, the number of automorphisms
of a smooth compact arithmetic ball quotient surface $X=\Gamma\backslash\mathbb{\mathbb{B}}_{2}$
is bounded by $288\cdot e(X)$, where $e(X)$ denotes the topological
Euler number of $X$. \\
Furthermore, Stover characterizes the arithmetic ball quotient surfaces
$X$ whose automorphism groups attain this bound, which by analogy
with Hurwitz curves, he calls \emph{Hurwitz ball quotients}; all such
surfaces are finite Galois coverings of the Deligne-Mostow orbifold
$\Lambda\backslash\mathbb{B}_{2}$ corresponding to the quintuple
$(2/12,2/12,2/12,7/12,11/12)$ (see \cite{Mostow, Stover}). \\
 Stover constructs also a Hurwitz ball quotient $S$ with Euler number
$e(S)=63$ and automorphism group $\Aut(S)$ isomorphic to $U_{3}(3)\times\mathbb{Z}/3\mathbb{Z}$,
of order $18144=2^{5}3^{4}7$. He shows that $S$ is the unique Hurwitz
ball quotient with Euler number $e=63$, and moreover that $e=63$
is the minimal possible value for the Euler number of a Hurwitz ball
quotient. Having this property the surface $S$ can be seen as the
2-dimensional analog of the Klein's quartic which is the unique curve
uniformized by the ball $\BB_{1}$ with minimal genus and maximal
possible automorphism group.\\
 Our aim is to study more closely the cohomology of this particular
surface $S$, which we will call \emph{Stover surface} in the following.
This surface $S$ has the following numerical invariants (see \cite{Stover}):

\smallskip
\begin{center}

\begin{tabular}{|c|c|c|c|c|c|}
\hline 
$e(S)$ & $H_{1}(S,\ZZ)$ & $q$ & $p_{g}=h^{2,0}$ & $h{}^{1,1}$  & $b_{2}(S)$\tabularnewline
\hline 
$63$ & $\ZZ^{14}$ & $7$ & $27$ & $35$  & $89$\tabularnewline
\hline 
\end{tabular}

\end{center}
\smallskip

Let $V$ be a vector space. Let us recall that a $2$-vector $w\in\wedge^{2}V$
has \textit{rank 1} or is \textit{decomposable} if there are vectors
$w_{1},w_{2}\in V$ with $w=w_{1}\wedge w_{2}$. A vector $w\in\wedge^{2}V$
has \textit{rank $2$} if there exist linearly independent vectors
$w_{i}\in V,\, i=1,..,4$ such that $w=w_{1}\wedge w_{2}+w_{3}\wedge w_{4}$. 

Let $B$ be an Abelian fourfold and let $p:S\to B$ be a map such
that $p(S)$ generates $B$. We say that $S$ is \textit{Lagrangian
with respect to} $p$ if there exists a basis $w_{1},\dots,w_{4}$
of $p^{*}H^{0}(B,\Omega_{B})$ such that the rank 2 vector $w=w_{1}\wedge w_{2}+w_{3}\wedge w_{4}$
is in the kernel of the natural map $\phi^{2,0}:\wedge^{2}H^{0}(S,\Omega_{S})\to H^{0}(S,K_{S})$.
\begin{thm}
\label{thm: MAIN}The surface $S$ has maximal Picard number. The
natural map 
\[
\phi^{1,1}:\, H^{0}(S,\Omega_{S})\otimes H^{1}(S,\mathcal{O}_{S})\to H^{1}(S,\Omega_{S})
\]
 is surjective with a $14$-dimensional kernel. The kernel of the
map
\[
\phi^{2,0}:\wedge^{2}H^{0}(S,\Omega_{S})\to H^{0}(S,K_{S})
\]
is $7$-dimensional and contains no decomposable elements. The set
of rank $2$ vectors in $Ker(\phi^{2,0})$ is a quadric hypersurface.\\
There exists an infinite number (up to isogeny) of maps $p:S\to B$
(where $B$ is an Abelian fourfold) such that $S$ is Lagrangian with
respect to $p$.\\
The Albanese variety of $S$ is isomorphic to $(\CC/\ZZ[\alpha])^{7}$,
for $\alpha=e^{2i\pi/3}$.
\end{thm}
By the Castelnuovo - De Franchis Theorem, the fact that there are
no decomposable elements in $\wedge^{2}H^{0}(S,\Omega_{S})$ means
that $S$ has no fibration $f:S\to C$ onto a curve of genus $g>1$.
Moreover Theorem \ref{thm: MAIN} implies that $S$ has the remarkable
feature that both maps 
\[
\begin{array}{c}
\phi^{2,0}:\,\wedge^{2}H^{1,0}(S)\to H^{2,0}(S)\\
\phi^{1,1}:\, H^{1,0}(S)\otimes H^{0,1}(S)\to H^{1,1}(S)
\end{array}
\]
have a non-trivial kernel. With Schoen surfaces (see \cite[Remark 2.6]{Ciliberto}),
this is the second example of surfaces enjoying such properties. For
more on this subject, see e.g. \cite{Barja, Bogomolov, BPS, Campana, Causin}. 

We obtain these results using Sullivan's theory on the second lower
quotient of the fundamental group $\pi_{1}(S)$ of $S$ (see \cite{BeauvilleFund}).

For the motivation and a historic account of surfaces with maximal
Picard number we refer to \cite{BeauvilleMax}.

\textbf{Aknowledgements} We are grateful to Marston Conder and Derek
Holt for their help in the computations of Theorem \ref{thm:Let--be Marston}.

\section{The Second lower central quotient of the fundamental group of $S$}

Let $\Pi:=\pi_{1}(X)$ be the fundamental group of a manifold $X$.
The group $H_{1}(X,\ZZ)$ is the abelianization of $\Pi$: $H_{1}(X,\ZZ)=\Pi/\Delta$
where $\Delta:=[\Pi,\Pi]$ is the derived subgroup of $\Pi$, that
is, the subgroup generated by all elements $[h,g]=g^{-1}h^{-1}gh$,
$h,g\in\Pi$.\\
 The second group in the lower central series $[\Delta,\Pi]$ is the
group generated by commutators $[h,g],$ with $h\in\Delta$, $g\in\Pi$.
It is a normal subgroup of the commutator group $\Delta$. According
to \cite{BeauvilleFund}, we have the following results:
\begin{prop}
(Sullivan) Let $X$ be a compact connected K\"{a}hler manifold. There
exists an exact sequence
\[
0\to\mbox{Hom}(\Delta/[\Delta,\Pi],\RR)\to\wedge^{2}H^{1}(X,\RR)\to H^{2}(X,\RR).
\]
(Beauville) Suppose $H_{1}(X,\ZZ)$ is torsion free. Then the group
$\Delta/[\Delta,\Pi]$ is canonically isomorphic to the cokernel of
the map 
\[
\mu:H_{2}(X,\ZZ)\to Alt^{2}(H^{1}(X,\ZZ))\,\, given\, by\,\,\mu(\sigma)(a,b)=\sigma\cap(a\wedge b),
\]
 where $Alt^{2}(H^{1}(X,\ZZ))$ is the group of skew-symmetric integral
bilinear forms on $H^{1}(X,\ZZ)$. 
\end{prop}
In the case of the Stover surface, computer calculations give us the
following result:
\begin{thm}
\label{thm:Let--be Marston}Let $\Pi=\pi_{1}(S)$ be the fundamental
group of the Stover surface and $\Delta=[\Pi,\Pi]$. The group $\Delta/[\Delta,\Pi]$
is isomorphic to $\ZZ/4\ZZ\times\ZZ^{28}$.\end{thm}
\begin{proof}
By the construction of $S$ \cite{Stover}, the fundamental group
$\Pi$ is isomorphic to the kernel $ker(\varphi)$ of the unique epimorphism
$\varphi:\Lambda\longrightarrow G$ from the Deligne-Mostow lattice
$\Lambda$ corresponding to the quintuple $(2/12,2/12,2/12,7/12,11/12)$
onto the finite group $G=U_{3}(3)\times\mathbb{Z}/3\mathbb{Z}$. The
lattice $\Lambda$ is described by Mostow in \cite{Mostow} as a complex
reflection group, and by generators and relation by Cartwright and
Steger in \cite{CS}. This lattice has presentation 
\[
\Lambda=\langle j,u,v,b|u^{4},v^{8},[u,j],[v,j],j^{-3}v^{2},uvuv^{-1}uv^{-1},(bj)^{2}(vu^{2})^{-1},[b,vu^{2}],b^{3},(bvu^{3})^{3}\rangle.
\]
MAGMA command \verb+LowIndexSubgroups+ is used to identify the unique
subgroup $\Gamma\triangleleft\Lambda$ of index $3$, which is $\Gamma=\langle u,jb,bj\rangle$.
Using the primitive permutation representation of $U_{3}(3)$ of degree
$28$, MAGMA is able to identify an homomorphism $\varphi$ from $\Gamma$
onto $U_{3}(3)$ induced from the assignment

\begin{align*}
u\mapsto&(3,8,23,20)(4,24,6,12)(7,9,14,22)(10,19,11,13)(15,16,21,18)(17,26,27,25)\\
jb\mapsto&(1,9,20,12,19,23,6,16)(2,27,14,17,13,26,15,25)(3,24)(4,5,10,21,7,11,28,8)\\
bj\mapsto&(1,13,20,15,19,2,6,14)(4,9,10,12,7,23,28,16)(5,27,21,17,11,26,8,25)(22,24).
\end{align*}

This homomorphism extends to an homomorphism $\varphi$ from $\Lambda$
onto $G$ such that $\Pi=ker(\varphi)$ is a torsion-free normal subgroup
in $\Lambda$, it is the fundamental group of $S$ (see \cite{Stover}).
Let be $\Delta=[\Pi,\Pi]$ and $\Delta_{2}=[\Delta,\Pi]$. It is easy
to check that that $\Delta_{2}$ is distinguished into $\Pi$. The
image of $\Delta$ under the quotient map $\Pi\longrightarrow\Pi/\Delta_{2}$
is $\Delta/\Delta_{2}$, but we observe that it is also equal the
commutator subgroup $[\Pi/\Delta_{2},\Pi/\Delta_{2}]$, and therefore,
the computation of $\Delta/\Delta_{2}$ is reduced to the one of the
derived group $[\Pi/\Delta_{2},\Pi/\Delta_{2}]$. \\
The MAGMA command \verb+g:=Rewrite(G,g)+ is used to have generators
and relations of both subgroups $\Gamma<\Lambda$ and $\Pi<\Gamma$.
The command \verb+NilpotentQuotient(.,2)+ applied to $\Pi$ describes
$\Pi/\Delta_{2}$ in terms of a polycyclic presentation. The derived
subgroup $[\Pi/\Delta_{2},\Pi/\Delta_{2}]$ is obtained with \verb+DerivedGroup(.)+
applied to $\Pi/\Delta_{2}$. Finally, applying the MAGMA function
\verb+AQInvariants+ to $[\Pi/\Delta_{2},\Pi/\Delta_{2}]$, MAGMA
computes that the structure of $\Delta/\Delta_{2}$ is $\mathbb{Z}/4\mathbb{Z}\times\mathbb{Z}^{28}$.
\end{proof}

\begin{cor}
\label{cor:Therefore-we-have} The dimension of the kernel of $\wedge^{2}H^{1}(S,\RR)\to H^{2}(S,\RR)$
is $28$.
\end{cor}

\section{Computation of the map $\wedge^{2}H^{1}(S,\CC)\to H^{2}(S,\CC)$}

Let $A$ be the Albanese variety of the Stover surface $S$. The invariants
are:

\[
\begin{array}{c}
H_{1}(A,\ZZ)=H_{1}(S,\ZZ)=\ZZ^{14},\ H_{2}(A,\ZZ)=\wedge^{2}H_{1}(A,\ZZ),\ H^{2,0}(A)=\wedge^{2}H^{1,0}(S)\\
H^{1,1}(A)=H^{1,0}(S)\otimes H^{0,1}(S),\ H^{0,2}(A)=\wedge^{2}H^{0,1}(S),
\end{array}
\]
and

\smallskip
\begin{center}

\begin{tabular}{|c|c|c|c|c|}
\hline 
$H_{1}(A,\ZZ)$ & $q$ & $h^{2,0}(A)$ & $h{}^{1,1}(A)$  & $b_{2}(A)$\tabularnewline
\hline 
$\ZZ^{14}$ & $7$ & $21$ & $49$  & $91$\tabularnewline
\hline 
\end{tabular}

\end{center}
\smallskip

We have a map respecting Hodge decomposition
\[
\begin{array}{c}
H^{2,0}(A)\oplus H^{1,1}(A)\oplus H^{0,2}(A)\\
\begin{array}{ccc}
\downarrow\,\,\,\,\,\,\, & \,\,\,\,\,\,\,\downarrow\,\,\,\,\,\,\, & \,\,\,\,\,\,\,\downarrow\end{array}\\
H^{2,0}(S)\oplus H^{1,1}(S)\oplus H^{0,2}(S)
\end{array},
\]
which is an equivariant map of $\Aut(S)$-modules. By Corollary \ref{cor:Therefore-we-have},
the kernel of that map is $28$ dimensional ; it is moreover a $\Aut(S)$-module. 

According to the Atlas tables \cite{Atlas}, the group $U_{3}(3)$
has $14$ irreducible representations $\chi_{i},\,1\leq i\leq14$
of respective dimension $1,6,7,7,7,14,21,21,21,27,28,28,32,32$. 

The irreducible representations of $\Aut(S)=U_{3}(3)\times\ZZ/3\ZZ$
are the $\chi_{i}^{t}$, $i=1,...,14,\, t=0,1,2$ where $(g,s)\in U_{3}(3)\times\ZZ/3\ZZ$
acts on the same space as $\chi_{i}$ with action $(g,s)\cdot v=\alpha^{s}g(v)$
with$\alpha=e^{2i\pi/3}$ a primitive third root of unity. 
\begin{thm}
\label{thm:The-image-of}The image of $S$ by the Albanese map $\vartheta:S\to A$
is $2$-dimensional. \\
The map $H^{1,1}(A)\to H^{1,1}(S)$ is surjective, with a $14$ dimensional
kernel isomorphic to $\chi_{6}^{0}$ as an $\Aut(S)$-module. We have
$H^{1}(S,\ZZ)=\chi_{3}^{1}\oplus\chi_{3}^{2}$ and $H^{1,1}(S)=\chi_{1}^{0}\oplus\chi_{3}^{0}\oplus\chi_{10}^{0}$,
as $\Aut(S)$-modules. \\
The kernel of the natural map $\wedge^{2}H^{0}(S,\Omega_{S})\to H^{0}(S,K_{S})$
is $7$-dimensional, isomorphic to $\chi_{3}^{0}$ as a $\Aut(S)$-module.\\
The surface $S$ has maximal Picard number. \\
The Albanese variety $A$ of $S$ is isomorphic to $(\CC/\ZZ[\alpha])^{7}$,
for $\alpha=e^{2i\pi/3}$.
\end{thm}
Since $A$ is CM, it follows that $S$ is Albanese standard, that
is, the class of its image inside its Albanese variety A sits in the
subring of $H^{*}(A,\QQ)$ generated by the divisor classes. That
contrasts with the above mentioned Schoen surfaces, see \cite{Ciliberto}.
\begin{proof}
Suppose that the image of $S$ in $A$ is $1$-dimensional. Then there
exists a smooth curve $C$ of genus $7$ with a fibration $f:S\to C$
and the map $\wedge^{2}H^{0}(S,\Omega_{S})\to H^{0}(S,K_{S})$ is
the $0$ map and the kernel of $\wedge^{2}H^{1}(S,\CC)\to H^{2}(S,\CC)$
is at least $42$ dimensional, which is impossible. Thus the image
of $S$ by the Albanese map $\vartheta:S\to A$ is $2$-dimensional.

According to the Atlas character table \cite{Atlas}, the possibilities
for the $U_{3}(3)$-module $H_{1}(S,\ZZ)=H_{1}(A,\ZZ)=\ZZ^{14}$ are:
\[
\chi_{3}^{\oplus2},\mathcal{R}_{\ZZ}(\chi_{4})=\mathcal{R}_{\ZZ}(\chi_{5})=\chi_{4}\oplus\chi_{5},\,\chi_{4}^{\oplus2},\,\chi_{5}^{\oplus2}\,\mbox{or}\,\chi_{6}
\]
where $\mathcal{R}_{\ZZ}(\chi_{j})$ is the restriction to $\ZZ$
of the $7$-dimensional complex representation $\chi_{j}$ defined
over $\ZZ[i]$. It cannot be $\chi_{4}^{\oplus2}$ nor $\chi_{5}^{\oplus2}$
because these are not is not defined over $\ZZ$ (some traces of elements
are in $\ZZ[i]\setminus\ZZ$). We cannot have $H^{1}(S,\ZZ)=\chi_{6}$
since $\chi_{6}$ remains irreducible, but $H^{1}(S,\ZZ)\otimes\CC=H^{1,0}\oplus H^{0,1}$
is a Hodge decomposition on which the representation of $U_{3}(3)$
splits. 

By duality, the kernel of $H^{2,0}(A)\to H^{2,0}(S)$ has same dimension
$d$ as the kernel of $H^{0,2}(A)\to H^{0,2}(S)$. Let $k$ be the
dimension of the kernel of the $U_{3}(3)$-equivariant map $H^{1,1}(A)\to H^{1,1}(S)$.
We have $28=k+2d$, moreover since $h^{1,1}(S)=35$ and $h^{1,1}(A)=49$,
we get $28\geq k\geq14$.

Let us suppose that $H^{1}(S,\ZZ)=\chi_{4}\oplus\chi_{5}$. Then the
representation $H^{1,1}(A)$ equals to $\chi_{4}\otimes\chi_{5}=\chi_{1}+\chi_{7}+\chi_{10}$
(of dimension $1+21+27$). An Abelian variety on which a finite group
$G$ acts possesses a $G$-invariant polarization (for example $\sum_{g\in G}g^{*}L$,
where $L$ is some polarization). Therefore the one dimensional $\Aut(S)$-invariant
space of $H^{1,1}(A)$ is generated by the class of an ample divisor
and the natural map $\vartheta^{*}:H^{1,1}(A)\to H^{1,1}(S)$ is injective
on that subspace. Therefore the map $\vartheta^{*}$ has a kernel
of dimension $k=21,\,27$ or $48$. This is impossible because $k+2d$
equals $28$.

Hence, we have $H^{1}(S,\ZZ)=\chi_{3}^{\oplus2}$ and moreover
\[
H^{2,0}(A)=\wedge^{2}\chi_{3}=\chi_{3}\oplus\chi_{6}
\]
 (the dimensions are $21=7+14$) and 
\[
H^{1,1}(A)=\chi_{3}^{\otimes2}=\chi_{1}\oplus\chi_{3}\oplus\chi_{6}\oplus\chi_{10}
\]
 ($49=1+7+14+27$). By checking the possibilities, we obtain $k=14$,
$H^{1,1}(S)=\chi_{1}\oplus\chi_{3}\oplus\chi_{10}$, and the map $H^{1,1}(A)\to H^{1,1}(S)$
is surjective. The kernel of the map $H^{2,0}(A)\to H^{2,0}(S)$ is
isomorphic to $\chi_{3}$, of dimension $7$, the action of $U_{3}(3)$
on $H^{2,0}(S)$ is then $H^{2,0}(S)=\chi_{6}\oplus\chi,$ where $\chi$
is a $13$ dimensional representation.

Let $\sigma\in Aut(S)=U_{3}(3)\times\ZZ/3\ZZ$ be the order $3$ automorphism
commuting with every other element. It corresponds to an element $\sigma'\in\Lambda$
normalizing $\Pi$ in $\Lambda$ and such that the group $\Pi'$ generated
by $\Pi$ and $\sigma'$ contains $\Pi$ with index $3$. Using MAGMA,
one find that we can choose $\sigma'=j^{4}$, where $j$ is the order
$12$ element described in the proof of Theorem \ref{thm:Let--be Marston}.
\\
The quotient surface $S/\sigma$ of $S$ by $\sigma$ is equal to
$\BB_{2}/\Pi'$. The fundamental group of $S'$ is $\Pi'/\Pi'_{tors}$
where $\Pi'_{tors}$ is the subgroup of $\Pi'$ generated by torsion
elements. Using MAGMA, one find that $\Pi'$ has a set of $8$ generators
with $7$ of them which are torsion elements. Using these elements,
we readily compute that $\Pi'/\Pi'_{tors}$ is trivial. Therefore
the space of one-forms on $S$ that are invariant by $\sigma$ is
$0$. Using the symmetries of $U_{3}(3)$, one see that $\sigma$
acts on the tangent space $H^{0}(S,\Omega_{S})^{*}$ as the multiplication
by $\alpha$ or $\alpha^{2}$. After possible permutation of $\sigma$
and $\sigma^{2}$, we can suppose it is $\alpha$.

We see that the representation of $Aut(S)$ on $H_{1}(S,\ZZ)$ is
$\chi_{3}^{1}\oplus\chi_{3}^{2}$. The lattice $H_{1}(S,\ZZ)\subset H^{0}(S,\Omega_{S})^{*}$
is moreover a $\ZZ[\alpha]$-module. The ring $\ZZ[\alpha]$ is a
principal ideal domain, therefore $H_{1}(S,\ZZ)=\ZZ[\alpha]^{7}$
(for the choice of a certain basis) and $A$ is isomorphic to $(\CC/\ZZ[\alpha])^{7}$.\\
Therefore $A$ has maximal Picard number and all the classes of $H^{1,1}(A)$
are algebraic. These classes remain of course algebraic under the
map $H^{1,1}(A)\to H^{1,1}(S)$, which is surjective. Thus $S$ is
a surface with maximal Picard number. 
\end{proof}

\section{Lagrangian surfaces and Stover surface}

Let $B$ be an Abelian fourfold and let $p:S\to B$ be a map such
that $p(S)$ generates $B$. Let us recall that $S$ is Lagrangian
with respect to $p$ if there exists a basis $w_{1},\dots,w_{4}$
of $p^{*}H^{0}(B,\Omega_{B})$ such that the rank 2 vector $w=w_{1}\wedge w_{2}+w_{3}\wedge w_{4}$
is in the kernel of the natural map $\phi^{2,0}:\wedge^{2}H^{0}(S,\Omega_{S})\to H^{0}(S,K_{S})$.
Let us now prove
\begin{thm}
\label{thm no decomp}The $7$ dimensional space $Ker(\phi^{2,0})$
contains no decomposable elements. The algebraic set of rank $2$
vectors in $Ker(\phi^{2,0})$ is a quadric $\tilde{Q}\subset Ker(\phi^{2,0})$.
\\
There exists an infinite number (up to isogeny) of maps $p:S\to B$
where $B$ is an Abelian fourfold such that $S$ is Lagrangian with
respect to $p$. \\
There exists an infinite number (up to isogeny) of maps $p:S\to B$
where $B$ is an Abelian fourfold such that 
\[
\tilde{Q}\cap p^{*}H^{0}(B,\wedge^{2}\Omega_{B})=\{0\},
\]
and for some of them we even have $Ker(\phi^{2,0})\cap p^{*}H^{0}(B,\wedge^{2}\Omega_{B})=\{0\}.$\\
 The generic rank 2 element $w$ in $\tilde{Q}\subset Ker(\phi^{2,0})$
does not correspond to any morphism to an Abelian fourfold. \end{thm}
\begin{proof}
We proved in Theorem \ref{thm:The-image-of} that 
\[
H^{2,0}(A)=\wedge^{2}\chi_{3}=\chi_{3}\oplus\chi_{6}
\]
and the kernel of $\phi^{2,0}:H^{2,0}(A)\to H^{2,0}(S)$ is the $7$-dimensional
subspace with representation $\chi_{3}$. In a basis $\gamma=(e_{1},\dots,e_{7})$
of $\chi_{3}=H^{0}(S,\Omega_{S})=H^{1,0}(S)$, the two following matrices
$A,B$ are generators of the group $U_{3}(3)$: 
\[
A=\left(\begin{array}{ccccccc}
-1 & 0 & 0 & 0 & 0 & 0 & 0\\
0 & -1 & 0 & 0 & 0 & 0 & 0\\
0 & 1 & 0 & 0 & 0 & 1 & 0\\
-1 & 0 & 0 & 0 & 0 & 0 & 1\\
0 & 0 & 0 & 0 & 1 & 0 & 0\\
0 & 1 & 1 & 0 & 0 & 0 & 0\\
-1 & 0 & 0 & 1 & 0 & 0 & 0
\end{array}\right),\, B=\left(\begin{array}{ccccccc}
0 & -1 & 0 & 0 & 0 & -1 & 0\\
0 & 1 & 1 & 0 & 0 & 0 & 0\\
0 & -1 & 0 & 0 & 0 & 0 & 0\\
1 & 0 & 0 & 0 & 0 & 0 & 0\\
0 & 0 & 0 & 0 & 0 & 0 & -1\\
0 & 0 & 0 & -1 & 0 & 0 & 1\\
0 & 0 & 0 & 0 & 1 & 0 & 1
\end{array}\right).
\]
Using the basis $\beta=(e_{ij})_{1\leq i<j\leq7}$ of $\wedge^{2}\chi_{3}$
($e_{ij}=e_{i}\wedge e_{j}$) with order $e_{ij}\leq e_{st}$ if $i<s$
or $i=s$ and $j\leq t$, one computes that the subspace $Ker(\phi^{2,0})=\chi_{3}\subset\wedge^{2}\chi_{3}$
is generated by the columns of the matrix $M\in M_{21,7}$, where
$^{t}M=(N,2I_{7})$, for 
\[
N=\left(\begin{array}{cccccccccccccc}
0 & 0 & 2 & -2 & -2 & -2 & 0 & 2 & 2 & -2 & 2 & 2 & 2 & 2\\
-1 & 0 & 0 & 2 & 4 & 0 & 1 & -3 & -3 & 1 & -3 & -4 & -2 & -4\\
0 & -2 & 0 & -2 & -2 & 0 & -2 & 2 & 2 & 0 & 0 & 2 & 2 & 2\\
-1 & -2 & 2 & 0 & -2 & 0 & -1 & 1 & 3 & 1 & 1 & 0 & 0 & 2\\
-1 & 1 & -1 & 3 & 1 & 3 & 0 & -4 & -2 & 2 & 0 & -4 & -2 & -2\\
0 & 3 & -3 & 1 & -1 & 1 & 1 & -3 & -3 & -1 & 1 & 0 & -2 & -2\\
1 & 1 & 1 & 3 & 3 & 1 & 2 & -2 & 0 & 0 & 0 & -2 & 0 & -2
\end{array}\right)\in M_{7,14}
\]
and $I_{7}$ the $7\times7$ identity matrix. Knowing that, we obtain
the ideal $I_{V}$ of the algebraic set $V$ of couples $(w_{1},w_{2})\in\chi_{3}\oplus\chi_{3}$
such that $w_{1}\wedge w_{2}\in Ker(\phi^{2,0})\subset\wedge^{2}\chi_{3}.$
That ideal is generated by $14$ homogeneous quadratic polynomials
in the variables $x_{1},\dots,x_{14}$. Let $W$ be the algebraic
set of couples $(w_{1},w_{2})\in\chi_{3}\oplus\chi_{3}$ such that
$w_{1}\wedge w_{2}=0\in\wedge^{2}\chi_{3}.$ The ideal $I_{W}$ of
$W$ is generated by the $2$ by $2$ minors of the matrix
\[
L=\left(\begin{array}{ccc}
x_{1} & \dots & x_{7}\\
x_{8} & \dots & x_{14}
\end{array}\right).
\]
Since $W\subset V$, we have $Rad(I_{V})\subset Rad(I_{W})$ where
$Rad(I)$ is the radical of an ideal $I$. On the other hand, using
Maple, one can check that the $21$ minors of $L$ are in $Rad(I_{V})$,
hence $Rad(I_{W})\subset Rad(I_{V})$, thus $V=W$. \\
We therefore conclude that the kernel of $\phi^{2,0}$ contains no
decomposable elements.

A 2-vector $w$ over a characteristic $0$ field can be expressed
uniquely as $w=\sum_{i,j}a_{ij}e_{i}\wedge e_{j}$ where $a_{ij}=\lyxmathsym{\textminus}a_{ji}$.
The rank of the vector $w$ is half the rank of the (skew-symmetric)
coefficient matrix $A_{w}:=(a_{ij})_{1\leq i,j\leq7}$ of $w$ \cite[Thm 1.7 \& Remark p. 13]{Bryant}.
Thus the $2$-vector $w=a_{1}v_{1}+\dots+a_{7}v_{7}$ in $Ker(\phi^{2,0})$
(where the $v_{i},\, i=1..7$ are the vectors corresponding to the
columns of the matrix $M$) is a rank $2$ vector if and only if the
$49$ $6\times6$ minors of the matrix $A_{w}$ are $0$. The radical
of the ideal generated by these minors is principal, generated by
a homogeneous quadric in $a_{1},\dots,a_{7}$ whose associated symmetric
matrix is 
\[
Q=\left(\begin{array}{ccccccc}
7 & 3 & 3 & 1 & -3 & -3 & -5\\
3 & 7 & 3 & 3 & 1 & -3 & -3\\
3 & 3 & 7 & 3 & 3 & 1 & -3\\
1 & 3 & 3 & 7 & 3 & 3 & 1\\
-3 & 1 & 3 & 3 & 7 & 3 & 3\\
-3 & -3 & 1 & 3 & 3 & 7 & 3\\
-5 & -3 & -3 & 1 & 3 & 3 & 7
\end{array}\right).
\]
Therefore $w\in Ker(\phi^{2,0})$ has rank $2$ if and only if $(a_{1},\dots,a_{7})Q^{t}(a_{1},\dots,a_{7})=0$.

The point $(10+8\alpha,-7,0,0,7,0,0)$ lies on the associated smooth
quadric $\tilde{Q}$, therefore $\tilde{Q}(\QQ[\alpha])$ is infinite.
Let be $w$ be a 2-vector in $\tilde{Q}(\QQ[\alpha])$. The decomposable
vector $\wedge^{2}w\not=0$ has coordinates in $\QQ[\alpha]$ in the
basis $(e_{i1}\wedge.\dots\wedge e_{i4})$ of $\wedge^{4}H^{0}(S,\Omega_{S})$.
The corresponding $4$-dimensional vector space $W$ is therefore
generated by $4$ vectors $w_{1},\dots,w_{4}$ with coordinates over
$\QQ[\alpha]$ in the basis $\gamma=(e_{1},\dots,e_{7})$ of $H^{0}(S,\Omega_{S})$.
\\
One computes that the image of $\QQ[\alpha][U_{3}(3)\times\ZZ/3\ZZ]$
in $M_{7}(\QQ[\alpha])$ is $49$ dimensional over $\QQ[\alpha]$,
thus 
\[
\QQ[U_{3}(3)\times\ZZ/3\ZZ]=M_{7}(\QQ(\alpha))\,(=End(A)\otimes\QQ)
\]
in the basis $\gamma$, ($H_{1}(S,\QQ[\alpha])$ is the $\QQ[\alpha]$-vector
space generated by $e_{1},\dots,e_{k}$) and therefore there exists
a morphism $p:S\to E^{4}=B$ (where $E=\CC/\ZZ[\alpha]$) such that
$W=p^{*}H^{0}(B,\Omega_{B})$. By hypothesis the image $p(S)$ generates
$B$. By construction 
\[
\wedge^{2}p^{*}H^{0}(B,\Omega_{B})\cap Ker(\phi^{2,0})
\]
is at least one dimensional since it contains $w$, and therefore
$S$ is Lagrangian for $p$.

A contrario, the trace of an order $2$ automorphism $\sigma\in\Aut(S)\subset\aut(A)$
acting on the tangent space of $A$ at $0$ equals to $-1$, therefore
the image $B'$ of the endomorphism $p:\sigma-1_{A}$, where $1_{A}$
is the identity of $A$ is an Abelian fourfold. Using Maple, one computes
that 
\[
\wedge^{2}p^{*}H^{0}(B,\Omega_{B})\cap Ker(f)=\{0\}.
\]
Let $\vartheta:S\to A$ be the Albanese map of $S$, and let $q:A\to A$
be an endomorphism with a $4$ dimensional image and a representation
in $M_{7}(\QQ)\subset M_{7}(\QQ(\alpha))$ in the basis $\gamma$.
Since the matrix $Q$ is positive definite, we have 
\[
\wedge^{2}p^{*}H^{0}(B,\Omega_{B})\cap\tilde{Q}=\{0\},
\]
where $p$ is the map $p=q\circ\vartheta:S\to B$. Therefore $S$
is not Lagrangian with respect to $p.$ \end{proof}
\begin{rem}
Let $X$ be a surface and let $\phi^{2,0}:\wedge^{2}H^{0}(X,\Omega_{X})\to H^{0}(X,K_{X})$
be the natural map. Let be $d=\dim Ker(\phi^{2,0})$ and $q=\dim H^{0}(X,\Omega_{X})$.
In the proof of Theorem \ref{thm no decomp}, we saw that the set
of rank $k$ vectors in $Ker(\phi^{2,0})$ is a determinantal variety:
the intersection of minors of size $\geq2k+1$ of some anti-symmetric
matrix of size $q\times q$ with linear entries in $d$ variables.
 It seems to the authors quite remarkable that for Stover's surface
the set of rank $2$ vectors (obtained as the zero set of $49$ $6\times6$
minors of a size $q=7$ matrix) is an hypersurface in $Ker(\phi^{2,0})$.
That hypersurface is the only $U_{3}(3)$-invariant quadric of $U_{3}(3)$
acting on $Ker(\phi^{2,0})$.
\end{rem}

 \vspace{0.2in}
 {\large{\setlength{\parindent}{0.2in} }}{\large \par}

{\large{Amir Džambi\'{c},}}{\large \par}

{\large{Johann Wolfgang Goethe Universität, Institut für Mathematik,}}{\large \par}

{\large{Robert-Mayer-Str. 6-8, 60325 Frankfurt am Main,}}{\large \par}

{\large{Germany}}{\large \par}

\texttt{\large{dzambic@math.uni-frankfurt.de}}{\large \par}

{\large{\vspace{0.2in}
 \setlength{\parindent}{0.5in}}}{\large \par}

{\large{Xavier Roulleau,}}{\large \par}

{\large{Laboratoire de Mathématiques et Applications, Université de
Poitiers,}}{\large \par}

{\large{Téléport 2 - BP 30179 - 86962 Futuroscope Chasseneuil }}{\large \par}

{\large{France}}{\large \par}

\texttt{\large{roulleau@math.univ-poitiers.fr}}
\end{document}